\documentclass[a4paper,10pt]{article}

\usepackage[utf8]{inputenc}
\usepackage[T1]{fontenc}
\usepackage{graphicx}
\usepackage{amsmath,amsfonts,amssymb,amsthm}
\usepackage{mathtools}
\usepackage{mathrsfs}
\usepackage[dvipsnames]{xcolor}
\usepackage{hyperref}
\usepackage{booktabs}
\usepackage{bbold}
\usepackage{bm}
\usepackage{color}
\usepackage[font=footnotesize]{caption}
\usepackage{enumitem}
\usepackage{empheq}
\usepackage{framed}
\usepackage{fullpage}
\usepackage{hyperref}
\usepackage{soul}
\usepackage{dirtytalk}

\mathtoolsset{showonlyrefs}

\definecolor{myred}{HTML}{880000}
\definecolor{mygreen}{HTML}{008800}
\definecolor{myblue}{HTML}{000088}
\definecolor{linkblue}{HTML}{0000BB}

\hypersetup{
  colorlinks,
  linkcolor={myred},
  citecolor={myblue},
  urlcolor={linkblue}
}

\DeclareMathOperator{\tr}{Tr}
\newcommand{\ind}[1]{\bm 1 \{ #1 \}}

\newtheorem{theorem}{Theorem}
\newtheorem*{theorem*}{Theorem}

\newtheorem{corollary}{Corollary}

\theoremstyle{definition}

\newtheorem{remark}{Remark}

\title{A remark on Kashin's discrepancy argument and partial coloring in the Koml\'os conjecture}
\author{
  Afonso S. Bandeira, Antoine Maillard, Nikita Zhivotovskiy\thanks{Department of Mathematics, ETH Z\"{u}rich, Switzerland, \href{}{\{bandeira, antoine.maillard, nikita.zhivotovskii\}@math.ethz.ch}}
}

\begin{document}

\maketitle

\begin{abstract}
In this expository note, we discuss an early partial coloring result of B. Kashin [C. R. Acad. Bulgare Sci., 1985]. Although this result only implies Spencer's six standard deviations [Trans. Amer. Math. Soc., 1985] up to a $\log\log n$ factor, Kashin's argument gives a simple proof of the existence of a constant discrepancy partial coloring in the setup of the Koml\'os conjecture.
\end{abstract}

A seminal result of Spencer \cite{spencer1985six} states that if $x_1, \ldots, x_n$ are vectors in $\{0, 1\}^n$, then there is a vector $\varepsilon \in \{1, -1\}^n$ such that
\begin{equation}
\label{eq:sixstandarddev}
\max\limits_{i \in \{1, \ldots, n\}}|\langle \varepsilon, x_i\rangle| \le 6\sqrt{n}.
\end{equation}
A few years later this result was shown independently by E.\ Gluskin \cite{gluskin1989extremal} who used convex geometry arguments.
In fact, the arguments of Gluskin improve upon earlier papers by B. Kashin \cite{kashin1985isometric, kashin1987trigonometric}, which imply Spencer's bound up to a multiplicative $\log \log n$ factor.

To the best of our knowledge, the work~\cite{kashin1985isometric} has not been translated.
In this expository note, we present this original argument and discuss some of its immediate applications. We provide a slightly more general statement\footnote{Namely, we replace the bound on the individual norms of the vectors by a bound on the sum of these norms squared.} whose proof follows from the arguments in \cite{kashin1985isometric}. 

\begin{theorem}[Kashin's partial coloring theorem \cite{kashin1985isometric}]
\label{thm:kashin}
There is an absolute constant $c > 0$ such that the following holds. Assume that $x_1, \ldots, x_m \in \mathbb{R}^n$ satisfy $\sum\nolimits_{i = 1}^m\|x_i\|_2^2 \le m$. There is a choice $\varepsilon \in \{-1, 0, 1\}^n$ such that $\|\varepsilon\|_{1} \ge n/6$ and 
\[
\max\limits_{i \in \{1, \ldots, m\}}|\langle \varepsilon, x_i\rangle| \le c\sqrt{\frac{m}{n}}~.
\]
\end{theorem}
In comparison, the classic partial coloring result of Gluskin \cite[Theorem 3]{gluskin1989extremal}, which implies Spencer's result \eqref{eq:sixstandarddev}, requires as stated $\max\nolimits_{i \in \{1, \ldots, m\}}\|x_i\|_2 \le 1$, instead of $\sum\nolimits_{i = 1}^m\|x_i\|_2^2 \le m$ required in Theorem \ref{thm:kashin}. Under this slightly stronger assumption, Gluskin's result claims the existence of $\varepsilon \in \{-1, 0, 1\}^n$ such that $\|\varepsilon\|_{1} \ge c_1n$ for some constant $c_1 \in (0, 1)$ and $\max\nolimits_{i \in \{1, \ldots, m\}}|\langle \varepsilon, x_i\rangle| \le c_2\sqrt{1 + \log \left(m/n\right)}$ for some constant $c_2 > 0$. Despite the worse dependence on $m/n$ in Theorem \ref{thm:kashin}, the partial coloring result of Theorem \ref{thm:kashin} implies \eqref{eq:sixstandarddev} up to a multiplicative $\log\log n$ factor.

\begin{corollary}[A full coloring argument of Kashin, Proposition 2 in \cite{kashin1987trigonometric}] There is an absolute constant $c > 0 $ such that the following holds. Assume that $x_1, \ldots, x_n \in \mathbb{R}^n$ each satisfy $\|x_i\|_{\infty} \le 1$. There is $\varepsilon \in \{-1, 1\}^n$ such that 
\[
\max\limits_{i \in \{1, \ldots, n\}}|\langle \varepsilon, x_i\rangle| \le c\sqrt{n}\log\log n~.
\]
\end{corollary}
\begin{proof}
We repeat the proof of Proposition 2 in \cite{kashin1987trigonometric}, whose translated version is available,  for the sake of completeness. Applying Theorem~\ref{thm:kashin} to $\{x_i/\sqrt{n}\}_{i = 1}^n$, we get a partial coloring $\varepsilon \in \{-1,0,1\}^n$ with at most $(5/6)n$ zeros satisfying
$\max_{i \in [n]}|\langle \varepsilon, x_i\rangle| \leq c \sqrt{n}$.
Let $S = \{j \in [n]: \varepsilon_j = 0\}.$ Denoting $n_1 = |S|$, we have $n_1 \leq (5/6)n$.
When we restrict the set of coordinates of $\{x_i\}_{i = 1}^n$ to $S$, we get vectors $x_i^{(1)} \in \mathbb{R}^{n_1}$, still satisfying $\|x_i^{(1)}\|_\infty \leq 1$.
We can thus apply Theorem~\ref{thm:kashin} recursively to $\{x_i^{(1)}/\sqrt{n_1}\}_{i = 1}^n$: at step $k$, we work with a set $S_k$ such that $n_k = |S_k| \leq (5/6)^k n$, and with
corresponding vectors $x_i^{(k)} \in \mathbb{R}^{n_k}$ for $i = 1, \ldots, n$.
After $s$ steps, we have at most $(5/6)^s n$ zero remaining coordinates in the coloring, and a total discrepancy (by the triangle inequality) of at most $cs \sqrt{n}$.
We set $s = \lceil\log \log n / \log (6/5)\rceil$, so that $n_s = n(5/6)^s \leq n / \log n$.
For the remaining $n_s$ coordinates we use a random value of $\varepsilon^{(s)} \in \{\pm 1\}^{n_s}$.
By the standard Hoeffding's inequality and union bound argument, the discrepancy of this random coloring is
\begin{align}
  \max_{i \in [n]} |\langle \varepsilon^{(s)}, x_i^{(s)}\rangle| \leq C \sqrt{\log n}
  \max_{i \in [n]} \|x_i^{(s)}\|_2 \leq C \sqrt{n_s \log n} \leq C \sqrt{n},
\end{align}
for a universal constant $C > 0$. Therefore, by the triangle inequality, the total discrepancy 
of the resulting coloring $\varepsilon \in \{\pm 1\}^n$ is 
\begin{align}
  \max_{i \in [n]} |\langle \varepsilon, x_i\rangle| \leq c \sqrt{n} \lceil \log \log n / \log (6/5) \rceil + C \sqrt{n} \leq C_2 \sqrt{n} \log \log n, 
\end{align}
for some universal constant $C_2 > 0$. The claim follows.
\end{proof}

A previously unnoticed\footnote{After this note was made public, the authors were made aware that a similar argument, based on Vaaler's theorem, was used by Y. Lonke to prove the result of Corollary \ref{cor:partialkomlos}. This proof was presented in the PhD dissertation of Y. Lonke (1998, Hebrew University, in Hebrew) \cite{lonkephd}. Remarkably, the argument of Y. Lonke is independent of B. Kashin's earlier application of Vaaler's Theorem.} corollary of Theorem \ref{thm:kashin} is the following partial coloring result. This result appears in the work of Spencer \cite[Theorem 16]{spencer1985six}, though its proof there is quite involved.

\begin{corollary}[Koml\'os partial coloring]
\label{cor:partialkomlos}
There is a constant $K > 0 $ such that the following holds. Assume that $u_1, \ldots, u_n \in \mathbb{R}^m$ satisfy $\max\nolimits_{i \in \{1, \ldots, n\}}\|u_i\|_2 \le 1$. There is $\varepsilon \in \{-1, 0, 1\}^n$ such that $\|\varepsilon\|_{1} \ge n/6$ and 
\[
\left\|\sum\nolimits_{i = 1}^n\varepsilon_i u_i\right\|_{\infty} \le K.
\]
\end{corollary}
\begin{proof}
Consider the vectors $a_{1}, \ldots, a_m \in \mathbb{R}^n$ defined by $a_j = (u_1^{(j)}, \ldots, u_n^{(j)})$ for $j = 1, \ldots, m$, where $u_k^{(j)}$ denotes the $j$-th coordinate of $u_k$.
We have $\sum\limits_{j = 1}^m \|a_j\|_2^2 = \sum\limits_{i = 1}^n \|u_i\|_2^2 \le n$, or equivalently $\sum\limits_{j = 1}^m \left\|a_j\sqrt{\frac{m}{n}}\right\|_2^2 \le m$.
Applying Theorem \ref{thm:kashin} to $\left\{a_j\sqrt{\frac{m}{n}}\right\}_{j = 1}^m$, we obtain for some $\varepsilon \in \{-1, 0, 1\}^n$ with $\|\varepsilon\|_{1} \ge n/6$:
\[
\left\|\sum\nolimits_{i = 1}^n\varepsilon_i u_i\right\|_{\infty} = \max\limits_{j \in \{1, \ldots, m\}}|\langle \varepsilon, a_j\rangle| = \sqrt{\frac{n}{m}}\max\limits_{j \in \{1, \ldots, m\}}\left|\left\langle \varepsilon, a_j\sqrt{\frac{m}{n}}\right\rangle\right| \le K.
\]
\end{proof}
The Koml\'{o}s conjecture asks if the result of Corollary \ref{cor:partialkomlos} holds for some full coloring $\varepsilon \in \{-1,1\}^n$ instead.
A different proof of Corollary \ref{cor:partialkomlos}, based on the ideas of Gluskin, appears in the work of Giannopoulos \cite{giannopoulos1997some}. Recent results on algorithmic discrepancy imply a similar result. However, instead of the \emph{true partial coloring} taking its values in $\{-1, 0, 1\}$, one usually considers \emph{fractional colorings} where the colors take values in $[-1, 1]$, though many of the colors are close to $\{-1, 1\}$ (see e.g., \cite{lovett2015constructive}). Regarding the fractional coloring results, we note that an earlier version of Theorem \ref{thm:kashin}, where $\varepsilon \in [-1, 1]^n$ and $\|\varepsilon\|_{1} \ge cn$ for some constant $c \in (0, 1)$ appears in \cite{Kashin1981} (see also \cite{gluskin1989extremal}).

\begin{remark}[Partial coloring iterations for Koml\'os]
We note that, for any $s$, iterating Corollary~\ref{cor:partialkomlos} implies a partial coloring with at most $n (5/6)^s$ zeros and discrepancy $K s$. For example, with $s$ scaling as $\log\log n$ this yields a partial coloring with only a $O(1/\log(n))$ fraction of zeros and discrepancy $O(\log \log n)$ and with $s$ scaling as $\log n$ one can get a full coloring with discrepancy $O(\log(n))$.
\end{remark}

We note that in the setup of Koml\'os, partial colorings are only known to imply the discrepancy bound $O(\log n)$,
whereas the best known discrepancy bound due to Banaszczyk \cite{banaszczyk1998balancing} scales as $O(\sqrt{\log n})$. The remainder of the note is devoted to the proof of Theorem \ref{thm:kashin}. 

\begin{theorem}[Vaaler's theorem \cite{vaaler1979geometric}]
\label{thm:vaalera}
Let $Q_d = [-1/2, 1/2]^d$, and let $A \in \mathbb{R}^{d \times n}$ satisfy $\operatorname{rank}(A) = n$. Then, the following inequality holds
\[
\frac{1}{\sqrt{\det(A^{\top}A)}} \le \int\nolimits_{\mathbb{R}^n}\ind{Ax \in Q_d}dx~.
\]
\end{theorem}
Let us provide some intuition behind this result. Let $P_{U}$ be a uniform distribution on $[-1/2, 1/2]$, and $P_{G}$ be a Gaussian measure with density $\exp(-\pi x^2)$. It is clear that for any closed, convex, symmetric set $A \subseteq [-1/2, 1/2]$, it holds that $P_{G}(A) \le P_U(A)$. In this case, we say that the measure $P_U$ is \emph{more peaked} than $P_G$. The proof of Theorem \ref{thm:vaalera} first shows that the uniform measure in $[-1/2, 1/2]^d$ is more peaked than the Gaussian measure with density $\exp(-\pi\|x\|_2^2)$. The proof is then generalized to an arbitrary dimension $d$, so that the final statement is essentially a comparison between the Gaussian measure and the volume of the convex polytope induced by $A$. Lower bounding the Gaussian measure has continued to be a fundamental approach in modern discrepancy. 

For $\lambda > 0$ consider the set 
\[
E_{\lambda} = \left\{z \in \mathbb{R}^n: \|z\|_{\infty} \le 1 \quad \textrm{and} \quad \max\limits_{i \in \{1, \ldots, m\}}|\langle z, x_i\rangle| \le \frac{1}{\lambda}\sqrt{\frac{m}{n}}\right\}~.
\]
We want to lower bound the volume of this set using the bound of Theorem \ref{thm:vaalera}. Let $A$ in this theorem be defined as follows. Set $d = m + n$. The first $m$ rows of $A$ are $\frac{\lambda x_i}{2}\sqrt{\frac{n}{m}}$, where $i = 1, \ldots, m$. The remaining $n$ rows are the normalized standard basis vectors $e_i/2$, where $i = 1, \ldots, n$. Clearly, the rank of $A$ is equal to $n$. It is straightforward to see that $\ind{Az \in Q_d} = \ind{z \in E_{\lambda}}$ for all $z \in \mathbb{R}^n$. Therefore,
\[
\int\nolimits_{\mathbb{R}^n}\ind{Ax \in Q_d}dx = \operatorname{vol}(E_{\lambda})~.
\]
Observe that $A^{\top}A \in \mathbb{R}^{n \times n}$ is positive semi-definite, and by the standard upper bound we have
\[
\det(A^{\top}A) \le \left(\frac{\tr(A^{\top}A)}{n}\right)^{n} = \left(\frac{\tr(AA^{\top})}{n}\right)^{n} = \left(\sum\limits_{i = 1}^{m} \frac{\lambda^2 \|x_i\|_2^2}{4}\frac{1}{m} + \frac{1}{4}\right)^{n} \le \left(\frac{\lambda^2}{4} + \frac{1}{4}\right)^{n}~.
\]
Fix $\delta = 0.01$ and choose $\lambda > 0$ small enough so that $(1+\lambda^2)^{-n/2} > (1 - \delta)^n$~. By Theorem \ref{thm:vaalera} we have
\[
2^n(1 - \delta)^n < \frac{1}{\sqrt{\det(A^{\top}A)}} \le \operatorname{vol}(E_{\lambda})~.
\]
In particular, for the set $(2 - \delta)E_{\lambda}$, we have $\operatorname{vol}((2 - \delta)E_{\lambda}) > 2^n(2 - \delta)^{n}(1 - \delta)^n$. This set is also convex, and symmetric around the origin. By Minkowski's convex body theorem \cite[Chapter 2, Ex. 1]{matousek2013lectures}, we have that the set $(2 - \delta)E_{\lambda}$ contains at least $2(2 - \delta)^n(1 - \delta)^n$ non-zero points with integer coordinates. By construction, all these integer points belong to $\{-1, 0, 1\}^n$. Finally, the size of the set $\{x \in \{-1, 0, 1\}^n: \|x\|_{1} \le n/6\}$ is 
\[
\sum\nolimits_{r = 0}^{\lfloor n/6\rfloor}\binom{n}{r}2^r \le 2^{\lfloor n/6\rfloor}\sum\nolimits_{r = 0}^{\lfloor n/6\rfloor}\binom{n}{r} \le  \left(12e\right)^{n/6}~.
\]
We check that for our choice of $\delta$, for all $n$ it holds that $\left(12e\right)^{n/6} \le 2(2 - \delta)^n(1 - \delta)^n$. Thus, by the counting argument, there is at least one $\varepsilon \in \{-1, 0, 1\}^n \cap (2-\delta)E_{\lambda}$ such that $\|\varepsilon\|_1 \ge n/6$. The claim of Theorem \ref{thm:kashin} follows.\qed

\paragraph{Acknowledgments.} The authors are thankful to Dylan Altschuler for encouraging us to make this note available, and to Boris Kashin and Yossi Lonke for valuable feedback.

\bibliographystyle{plainurl}
{\footnotesize
\bibliography{mybib}

\begin{thebibliography}{10}

\bibitem{banaszczyk1998balancing}
Wojciech Banaszczyk.
\newblock Balancing vectors and gaussian measures of n-dimensional convex
  bodies.
\newblock {\em Random Structures \& Algorithms}, 12(4):351--360, 1998.

\bibitem{giannopoulos1997some}
Apostolos~A Giannopoulos.
\newblock On some vector balancing problems.
\newblock {\em Studia Mathematica}, 122:225--234, 1997.

\bibitem{gluskin1989extremal}
Efim~Davydovich Gluskin.
\newblock Extremal properties of orthogonal parallelepipeds and their
  applications to the geometry of {B}anach spaces.
\newblock {\em Mathematics of the USSR-Sbornik}, 64(1):85, 1989.

\bibitem{Kashin1981}
Boris~Sergeevich Kashin.
\newblock On the diameters of {Sobolev} classes of small smoothness.
\newblock {\em Vestn. Mosk. Univ., Ser. I}, 36(5):50--54, 1981.
\newblock (in {R}ussian).

\bibitem{kashin1985isometric}
Boris~Sergeevich Kashin.
\newblock An isometric operator in ${L}^2(0, 1)$.
\newblock {\em Comptes rendus de l'Acad\'emie Bulgare des Sciences},
  38:1613--1616, 1985.
\newblock (in {R}ussian).
\newblock URL:
  \url{https://homepage.mi-ras.ru/~kashin/download/32_1985_dban_rus.pdf}.

\bibitem{kashin1987trigonometric}
Boris~Sergeevich Kashin.
\newblock On trigonometric polynomials with coefficients $+1, -1, 0$.
\newblock {\em Proceedings of the Steklov Institute of Mathematics},
  180:152--154, 1989.
\newblock In Proceedings of the conference on the theory of functions dedicated
  to the 80-th anniversary of Academician S. M. Nikolsky (Dnepropetrovsk, May
  29–June 1, 1985).

\bibitem{lonkephd}
Yossi Lonke.
\newblock {\em Combinatorial problems in finite dimensional normed spaces}.
\newblock PhD thesis, Hebrew University at Jerusalem, 1998.
\newblock (in {H}ebrew).

\bibitem{lovett2015constructive}
Shachar Lovett and Raghu Meka.
\newblock Constructive discrepancy minimization by walking on the edges.
\newblock {\em SIAM Journal on Computing}, 44(5):1573--1582, 2015.

\bibitem{matousek2013lectures}
Ji\v{r}i Matou\v{s}ek.
\newblock {\em Lectures on Discrete Geometry}, volume 212.
\newblock Springer Science \& Business Media, 2013.

\bibitem{spencer1985six}
Joel Spencer.
\newblock Six standard deviations suffice.
\newblock {\em Transactions of the American mathematical society},
  289(2):679--706, 1985.

\bibitem{vaaler1979geometric}
Jeffrey Vaaler.
\newblock A geometric inequality with applications to linear forms.
\newblock {\em Pacific Journal of Mathematics}, 83(2):543--553, 1979.

\end{thebibliography}
}
\end{document}